\documentclass[12pt, a4paper]{article}
\title{Matrix roots of order n of a multivariate polynomial}
\usepackage{authblk}
\small{\author[1]{Jordan Noumi Bakop}}
\small{\author[2]{Yves Baudelaire Fomatati}}

{\affil[1]{Department of Mathematics, Faculty of Science, Yaound\'{e} I University, Cameroon, \textit{jordanoumi@gmail.com}.}}

{\affil[2]{Department of Mathematics, The University of Bamenda, Cameroon, \textit{ fomatati.yves@uniba.cm}.}}

\date{}
\usepackage{latexsym}
\usepackage{amsmath,amssymb}
\usepackage[square,comma,numbers,sort&compress]{natbib}
\usepackage{geometry,mathtools}
\usepackage{resizegather}

\usepackage{leqno}
\usepackage{amsfonts}

\usepackage{enumerate}

\usepackage{amsthm}
\usepackage[utf8,latin9]{inputenc}

\theoremstyle{plain}
\newtheorem{remark}{Remark}[section]
\theoremstyle{plain}
\newtheorem{lemma}{Lemma}[section]
\theoremstyle{plain}
\newtheorem{proposition}{Proposition}[section]
\theoremstyle{plain}
\newtheorem{theorem}{Theorem}[section]
\theoremstyle{plain}
\newtheorem{definition}{Definition}[section]
\theoremstyle{plain}

\theoremstyle{plain}

\theoremstyle{plain}

\newtheorem{example}{Example}[section]

\theoremstyle{plain}

\usepackage[all, 2cell]{xy}
\usepackage{txfonts}

\begin{document}
\maketitle
\begin{quote}
  \textbf{Abstract}
\end{quote}

In this paper, we construct the matrix nth root of a multivariate polynomial with coefficients in a field K which possesses an nth root of unity. Moreover, we explain why this concept is interesting. Finally, we give several examples to illustrate our construction.
\\\\In the sequel, $K$ is a field.

	\section{Introduction}
	Let $f\in K[X_1,...,X_n]$ be a multivariate polynomial with coefficients in $K$ and indeterminates $X_1,...,X_n$. We further assume that $K$ has a primitive nth root of unity.
	\begin{definition} \cite{diveris2019matrix}
		A matrix square root of order $m$ of a polynomial  $f\in K[X_1,...,X_n]$ is an $m \times m$ matrix $A$ such that
		\[
		A^{2} = f \, I_m
		\]
		$I_m$ is the $m \times m$ identity matrix.
	\end{definition}
	
	In the sequel, we generalize this result for $n \geq 3$. The case $n=2$ was treated in \cite{diveris2019matrix}. The main result in \cite{diveris2019matrix} is that \textit{every polynomial $f\in K[X_1,...,X_n]$ has a matrix square root}.
	
	\section{Main results}
We present the main results in this section. We are going to show that every multivariate polynomial $f\in K[X_1,...,X_n]$ with coefficients in $K$ and indeterminates $X_1,...,X_n$ has a matrix nth root.  \\
We first define the concept of a \textit{manomial} in order to facilitate our discussion. 	
	
	\begin{definition}
	A \textbf{manomial} is a product of several variables and/or scalars raised to powers that could be different.
	\end{definition}
	
	\begin{example}
		\quad $x^4 y z^5, \quad 8x, \quad 2x^3$.
	\end{example}
	
	\begin{proposition}\label{rac_cub1}
		\quad If $f$ is a polynomial defined by
		\[
		f = x_1^{\alpha_1} x_2^{\alpha_2} x_3^{\alpha_3}
		\]
		with manomials $x_1, x_2, x_3$, then $f$ admits a matrix third root of unity of order 3.
	\end{proposition}
	
	\begin{proof}
		\quad Fix one of the manomials $x_i^{\alpha_i}$ \\
		(Without loss of generality let us suppose $x_1^{\alpha_1}$ has been fixed) \\
		Let $J_2$ be the matrix obtained by replacing the coefficients ``1'' of the identity matrix $I_2$ by the $x_j^{\alpha_j}$, $j \neq 1$. Thus, we take:
		\[
		J_2 =
		\begin{pmatrix}
			x_2^{\alpha_2} & 0 \\
			0 & x_3^{\alpha_3}
		\end{pmatrix}
		\]
		
		We can now let
		\[
		A_3 =
		\begin{pmatrix}
			0 & 0 & x_1^{\alpha_1} \\
			x_2^{\alpha_2} & 0 & 0 \\
			0 & x_3^{\alpha_3} & 0
		\end{pmatrix}
		\]
		
		therefore:
		
		\[
		\begin{pmatrix}
			0 & 0 & x_1^{\alpha_1} \\
			x_2^{\alpha_2} & 0 & 0 \\
			0 & x_3^{\alpha_3} & 0
		\end{pmatrix}^3
		= x_1^{\alpha_1} x_2^{\alpha_2} x_3^{\alpha_3} \, I_3
		\]
		
		\[
		(A_3)^3 = f \cdot I_3 \quad \text{with } f = x_1^{\alpha_1} x_2^{\alpha_2} x_3^{\alpha_3} \text{ and } I_3 \text{ the identity matrix of order } 3.
		\]	
	\end{proof}
	
	\begin{example}
		\begin{enumerate}
			\item Let $f = x^2 y z = (x^2)(y)(z)$ \\
			\[
			A = \begin{pmatrix}
				0 & 0 & x^2 \\
				y & 0 & 0 \\
				0 & z & 0
			\end{pmatrix}
			A^2 = \begin{pmatrix}
				0 & x^2 z & 0 \\
				0 & 0 & x^2 y \\
				yz & 0 & 0
			\end{pmatrix}
			A^3 = \begin{pmatrix}
				x^2 y z & 0 & 0 \\
				0 & x^2 y z & 0 \\
				0 & 0 & x^2 y z
			\end{pmatrix}
			\]
			$A$ is indeed a matrix cube root of $f$
			\item Let $g = x^2tz^3y = (x^2 t)(z)(z^2 y)$ \\
			\[
			B = \begin{pmatrix}
				0 & 0 & x^2 t \\
				z & 0 & 0 \\
				0 & y z^2 & 0
			\end{pmatrix}
			B^2 = \begin{pmatrix}
				0 & x^2 y z^2 t & 0 \\
				0 & 0 & x^2 z t \\
				y z^3 & 0 & 0
			\end{pmatrix}
			B^3 = \begin{pmatrix}
				x^2 y z^3 t & 0 & 0 \\
				0 & x^2 y z^3 t & 0 \\
				0 & 0 & x^2 y z^3 t
			\end{pmatrix}
			\]
			$B$ is indeed a matrix cube root of $g$.
		\end{enumerate}	
	\end{example}

	\begin{proposition}
		Let $A=(a_{ij})$ be a matrix.\\
		\[
		A^n = \sum_{i,j} a_{ij}^{(n)} E_{ij}
		\]
		where $a_{ij}^{(n)}$ is the $(i,j)$ coefficient of the matrix $A^n$ and $E_{ij}$ is the elementary matrix.\\
		
		\[
		a_{ij}^{(n)} = \sum_{k=1}^{n} a_{ik} a_{kj}^{(n-1)}
		\]
		
		\[
		a_{ij}^{(1)} = a_{ij}
		\]
		
	\end{proposition}
	
	\begin{theorem}\label{rac_cub2}
		Let $n \geq 3$. \\
		Any polynomial $f = x_1^{\alpha_1} \cdots x_n^{\alpha_n}$ where $x_i$ are manomials possesses a matrix nth root of order $n$.
	\end{theorem}
	
	\begin{proof}
		Let $A = (a_{ij})$ where
		\[
		\begin{cases}
			a_{1,n} = x_1^{\alpha_1} \\
			a_{i,i-1} = x_i^{\alpha_i} \quad \text{pour } 2 \leq i \leq n \\
			0 \quad \text{everywhere else}
		\end{cases}
		\]
		In other words,
		\[
		A =
		\begin{pmatrix}
			0 & x_1^{\alpha_1} \\
			J_{n-1} & 0
		\end{pmatrix}
		\]
		
		where $J_{n-1}$ is obtained by replacing the coefficients $"1"$ by the $x_j^{\alpha_j},\ j \neq 1$.
		\begin{enumerate}
			\item \[
			\begin{aligned}
				a_{11}^{(n)} &= \sum_{k=1}^{n} a_{1k} a_{k1}^{(n-1)} \\
				&= a_{1,n} \cdot a_{n\,n-1} \cdots a_{3,2}\, a_{2,1} \\
				&= x_1^{\alpha_1} \cdots x_n^{\alpha_n}.
			\end{aligned}
			\]
			
			\item Let $l \neq 1$,
			\[
			\begin{aligned}
				a_{ll}^{(n)} &= \sum_{k=1}^{n} a_{lk} a_{k,l}^{(n-1)} \\
				&= a_{l,l-1}\, a_{l-1\,l}^{(n-1)} \\
				&= \left( \prod_{p=2}^{l} a_{p,p-1} \right) \times a_{1,n} \times a_{n,l}^{(n-l)} \\
				&= \left( \prod_{p=2}^{l} a_{p,p-1} \right) \times a_{1,n} \times a_{n,n-1} \times \cdots \times a_{l+1,l} \\
				&= x_1^{\alpha_1} \cdots x_n^{\alpha_n}.
			\end{aligned}
			\]
			
			\item Let $l, p$ be such that $l \neq p$
			\[
			\begin{aligned}
				a_{l,p}^{(n)} &= \sum_{k=1}^{n} a_{l,k} a_{k,p}^{(n-1)} \\
				&= \left( \prod_{t=2}^{l} a_{t,t-1} \right) \times a_{1,n}\, a_{n,p}^{(n-l)} \\
				&= \left( \prod_{t=2}^{l} a_{t,t-1} \right) \times a_{1,n} \times a_{n,n-1} \times \cdots \times a_{l+1,p}.
			\end{aligned}
			\]
			and $a_{l+1,p}=0$ by definition of the matrix $A$.
			
			Hence,
			\[
			(A)^n = x_1^{\alpha_1} \cdots x_n^{\alpha_n}\, I_n
			\]
		\end{enumerate}	
	\end{proof}

	\begin{example}
		\begin{enumerate}
			\item 4th root \\
			$h = x^4 y t^2 = (x^4)(y)(t^2)(1)$. Here $\alpha$ is the primitive 4th root of unity. \\
			Let
				\[
			D = \begin{pmatrix}
				0 & 0 & 0 & x^4 \\
				y & 0 & 0 & 0 \\
				0 & t^2 & 0 & 0 \\
				0 & 0 & 1 & 0
			\end{pmatrix}
			D^2 = \begin{pmatrix}
				0 & 0 & x^4 & 0 \\
				0 & 0 & 0 & x^4 y \\
				t^2 y & 0 & 0 & 0 \\
				0 & t^2 & 0 & 0
			\end{pmatrix}
			D^4 = \begin{pmatrix}
				x^4 t^2 y & 0 & 0 & 0 \\
				0 & x^4 t^2 y & 0 & 0 \\
				0 & 0 & x^4 t^2 y & 0 \\
				0 & 0 & 0 & x^4 t^2 y
			\end{pmatrix}
			\]
			
			So, $B$ is indeed a matrix fourth root of $f$.
			\item 5th root\\
			$f = x y^3 z t = (x)(y)(y^2)(z)(t)$. Here $\alpha$ is the primitive fifth root of unity. \\
			\[
			A = \begin{pmatrix}
				0 & 0 & 0 & 0 & x \\
				y & 0 & 0 & 0 & 0 \\
				0 & y^2 & 0 & 0 & 0 \\
				0 & 0 & z & 0 & 0 \\
				0 & 0 & 0 & t & 0
			\end{pmatrix} ;
			A^2 = \begin{pmatrix}
				0 & 0 & 0 & xt & 0 \\
				0 & 0 & 0 & 0 & xy \\
				y^3 & 0 & 0 & 0 & 0 \\
				0 & y^2 z & 0 & 0 & 0 \\
				0 & 0 & zt & 0 & 0
			\end{pmatrix} ;
			A^3 = \begin{pmatrix}
				0 & 0 & xtz & 0 & 0 \\
				0 & 0 & 0 & xyt & 0 \\
				0 & 0 & 0 & 0 & xy^3 \\
				y^3 z & 0 & 0 & 0 & 0 \\
				0 & y^2 z t & 0 & 0 & 0
			\end{pmatrix} ;
			\]
			\[
			A^4 = \begin{pmatrix}
				0 & xy^2zt & 0 & 0 & 0 \\
				0 & 0 & xyzt & 0 & 0 \\
				0 & 0 & 0 & xy^3t & 0 \\
				0 & 0 & 0 & 0 & xy^3z \\
				y^3zt & 0 & 0 & 0 & 0
			\end{pmatrix} ;
			A^5 = x y^3 z t \begin{pmatrix}
				1 & 0 & 0 & 0 & 0 \\
				0 & 1 & 0 & 0 & 0 \\
				0 & 0 & 1 & 0 & 0 \\
				0 & 0 & 0 & 1 & 0 \\
				0 & 0 & 0 & 0 & 1
			\end{pmatrix}
			\]
			Thus, $A$ is indeed a matrix fifth root of $f$.
		\end{enumerate}	
	\end{example}

	\begin{theorem}
		Let
		\[
		f_k = g_1 h_1 l_1 + \cdots + g_k h_k l_k \quad (k \geq 2)
		\]
		and each $g_i, h_i, l_i$ is a manomial. We assume that there exists a primitive cube root   $\alpha$ in $K$.
		Then $f_k$ possesses a matrix cube root of order $3^k$.
	\end{theorem}
	
	\begin{proof}
		\begin{itemize}
			\item For $k=1$, let $f_1 = g_1 h_1 l_1$ which has a matrix nth root according to  proposition $\ref{rac_cub1}$
		\end{itemize}
		
		\begin{itemize}
			\item Assume a matrix cube root of $f_{k-1}$ d'ordre $3^{k-1}$ has been constructed, call it $A_{k-1}$.
		\end{itemize}
		Let
		\[
		A_k =
		\begin{pmatrix}
			A_{k-1} & X_1 & 0 \\
			0 & \alpha A_{k-1} & X_2 \\
			X_3 & 0 & \alpha^2 A_{k-1}
		\end{pmatrix}
		\]
		where $X_1 = g_k I_{3^{k-1}}$, $X_2 = h_k I_{3^{k-1}}$, $X_3 = l_k I_{3^{k-1}}$, and $\alpha$ is a primitive cube root of unity.
		   \[
		(A_k)^3 =  \begin{pmatrix}
	     A_{k-1}^3 + X_1 X_2 X_3 & (1+\alpha+\alpha^2) X_1 A_{k-1}^2 & (1+\alpha+\alpha^2) X_1 X_2 A_{k-1} \\
			(1+\alpha+\alpha^2) X_2 X_3 A_{k-1} & \alpha^3 A_{k-1}^3 + X_2 X_3 X_1 & (\alpha^2 + \alpha^3 + \alpha^4) X_2 A_{k-1}^2 \\
			(1 + \alpha^2 + \alpha^4) A_{k-1}^2 X_3 & (1 + \alpha + \alpha^2) X_3 X_1 A_{k-1} & \alpha^6 A_{k-1}^3 + X_3 X_1 X_2
		     \end{pmatrix}
		     \]
		
		     \[
		    (A_k)^3 =   \begin{pmatrix}
			(A_{k-1}^3 + X_1 X_2 X_3) I_{3^{k-1}} & 0 & 0 \\
			0 & (A_{k-1}^3 + X_1 X_2 X_3) I_{3^{k-1}} & 0 \\
			0 & 0 & (A_{k-1}^3 + X_1 X_2 X_3) I_{3^{k-1}}
		\end{pmatrix}
		\]
		\[
		(A_k)^3 = (f_{k-1} + g_k h_k l_k) \, I_{3 \times 3^{k-1}} = f_k \, I_{3^k}
		\]
		Hence, $f_k$ has a matrix cube root of size $3^k$.
	\end{proof}
	
	\begin{example}
		\begin{enumerate}
			\item Consider $g = x y z + xt$. Here $\alpha$ is a matrix cube root of unity. \\
			Let
			\[
			A_1 = \begin{pmatrix}
				0 & 0 & x \\
				y & 0 & 0 \\
				0 & z & 0
			\end{pmatrix} ;
			(A_1)^2 = \begin{pmatrix}
				0 & xz & 0 \\
				0 & 0 & xy \\
				yz & 0 & 0
			\end{pmatrix} ;
			(A_1)^3 = \begin{pmatrix}
				xyz & 0 & 0 \\
				0 & xyz & 0 \\
				0 & 0 & xyz
			\end{pmatrix}
			= xyz \, I_3
			\]
			
			\[
			A_2 =
			\begin{pmatrix}
				A_1 & x I_3 & 0 \\
				0 & \alpha A_1 & t I_3 \\
				I_3 & 0 & \alpha^2 A_1
			\end{pmatrix} ;
			(A_2)^3 = \begin{pmatrix}
				A_1^3 + xt I_3 & x(1+ \alpha + \alpha^2)A_1^2 & xt(1+ \alpha + \alpha^2)A_1 \\
				t(1+ \alpha + \alpha^2)A_1 & \alpha^3 A_1^3 + xt I_3 & \alpha^2 t(1+ \alpha + \alpha^2)A_1^2 \\
				(1+ \alpha^2 + \alpha^4)A_1^2 & x(1+ \alpha + \alpha^2)A_1 & \alpha^6 A_1^3 + xt I_3
			\end{pmatrix}
			\]
			 $A_2$ is indeed a matrix cube root of $g$.
			
			\item Consider $h = x^2 y + 3x t^2 + z^3$. Here, $\alpha$ is a matrix cube root of unity.
			
			Let
			
			\[
			D_1 = \begin{pmatrix}
				0 & 0 & x \\
				x & 0 & 0 \\
				0 & y & 0
			\end{pmatrix} ;
			(D_1)^3 = \begin{pmatrix}
				x^2y & 0 & 0 \\
				0 & x^2y & 0 \\
				0 & 0 & x^2y
			\end{pmatrix}
			= x^2y \, I_3
			\]

			\[
			D_2 =
			\begin{pmatrix}
				D_1 & 3x I_3 & 0 \\
				0 & \alpha D_1 & z I_3 \\
				z I_3 & 0 & \alpha^2 D_1
			\end{pmatrix}
			(D_2)^3 = \begin{pmatrix}
				D_1^3 + 3xz I_3 & 3x (1+ \alpha + \alpha^2) D_1^2 & 3xz(1+ \alpha + \alpha^2) D_1 \\
				z^2(1+ \alpha + \alpha^2)D_1 & \alpha^3 D_1^3+3xzI_3 & \alpha^2 z(1+\alpha + \alpha^2)D_1^2 \\
				z(1+ \alpha^2 + \alpha^4)D_1^2 & 3xz(1+ \alpha + \alpha^2)D_1 & \alpha^6 D_1^3 + 3xz I_3
			\end{pmatrix}
			\]

			\[
			D_3 =
			\begin{pmatrix}
				D_2 & z I_9 & 0 \\
				0 & \alpha D_2 & z I_9 \\
				z I_9 & 0 & \alpha^2 D_2
			\end{pmatrix}
			(D_3)^3 = \begin{pmatrix}
				D_2^3 + z^3 I_9 & z(1+ \alpha + \alpha^2)D_2^2 & z^2(1+ \alpha + \alpha^2)D_2 \\
				z^2(1+ \alpha + \alpha^2)D_2 & \alpha^3 D_2^3 + z^3 I_9 & \alpha^2 z(1+ \alpha + \alpha^2)D_2^2 \\
				z(1+ \alpha^2 + \alpha^4)D_2^2 & z^2(1+ \alpha + \alpha^2)D_2 & \alpha^6 D_2^3 + z^3 I_9
			\end{pmatrix}
			\]
			$D_3$ is indeed a matrix cube root of $h$.
		\end{enumerate}	
	\end{example}

	\begin{remark}
		If $A$ is a matrix nth root of a polynomial $f$, then $A^T$ is also a matrix nth root of $f$.
	\end{remark}
	
	\begin{definition}[\cite{kac2002quantum}]
		Let $A$ and $B$ be two matrices of the same size. If $AB = q BA$, then
		\[
		(A+B)^n = \sum_{k=0}^{n} \begin{pmatrix} n \\ k \end{pmatrix}_q A^k B^{n-k}
		\]
		$\begin{pmatrix} n \\ k \end{pmatrix}_q$ is the $q$-binomial coefficient.
	\end{definition}

	\begin{lemma}[\cite{kac2002quantum}]\label{binomial}
		If $q$ is an nth primitive root of unity, then
		\[
		(A+B)^n = A^n + B^n
		\]
	\end{lemma}

	\begin{theorem}
		Let $f_k = u_1 + \cdots + u_k$, $k \geq 2$, and each $u_i$ is a product of manomials. If there exists a primitive nth root of unity $\alpha$ in $\mathbb{K}$, then
		$f_k$ possesses a matrix nth root of order $n^k$, $(n \geq 3)$.
	\end{theorem}

	\begin{proof}
		For $k=1$, $f_1 = u_1$ possesses a matrix nth root of order $n^1=n$ according to theorem $\ref{rac_cub1}$ \\
		
		Assume there is a matrix nth root $A_{k-1}$ of $f_{k-1}$ of order $n^{k-1}$. \\
		Let
		\[
		A_k =
		\begin{pmatrix}
			A_{k-1} & X_1 & 0 & \cdots & 0 \\
			0 & \alpha A_{k-1} & X_2 & \cdots & 0 \\
			\vdots & \vdots & \ddots & \ddots & 0 \\
			0 & 0 & 0 & \alpha^{n-2} A_{k-1} & X_{n-1}  \\
			X_n & 0 & 0 & 0 & \alpha^{n-1} A_{k-1}
		\end{pmatrix}
		\]
		
	    $A_k = P_k + D_k$ with:
		\[
		P_k =
		\begin{pmatrix}
			0 & X_1 & 0 & \cdots & 0 \\
			0 & 0 & X_2 & \cdots & 0 \\
			\vdots & \vdots & \ddots & \ddots & 0 \\
			0 & 0 & 0 & 0 & X_{n-1}  \\
			X_n & 0 & 0 & 0 & 0
		\end{pmatrix}
	      D_k =
	      \begin{pmatrix}
	      	A_{k-1} & 0 & 0 & \cdots & 0 \\
	      	0 & \alpha A_{k-1} & 0 & \cdots & 0 \\
	      	\vdots & \vdots & \ddots & \ddots & 0 \\
	      	0 & 0 & 0 & \alpha^{n-2} A_{k-1} & 0  \\
	      	0 & 0 & 0 & 0 & \alpha^{n-1} A_{k-1}
	      \end{pmatrix}
		\]
		
		We have $P_k D_k = \alpha D_k P_k$. Since $\alpha$ is a primitive nth root of unity according to lemma  $\ref{binomial}$ (Observe that $P_k$ is the transpose of the matrix constructed in the proof of theorem $\ref{rac_cub2}$).
		\[
		A_k^n = (P_k + D_k)^n = P_k^n + D_k^n.
		\]
		
		Now $P_k^n = I_n \otimes(X_1 X_2 \cdots X_n) = I_n \otimes(u_k I_{n^{k-1}}) = u_k I_{n^k}$ et $(D_k)^n = f_{k-1} I_{n^k}$.
		
		So $A_k^n = f_k I_{n^k}$, which proves that $f_k$ admits a matrix nth root of order $n^k$.
	\end{proof}

	\begin{example}
			 Consider $f = xyz + 4x^2$. We want to find the matrix fourth root of $f$. We write
			$f = (x)(y)(z)(1) + (4)(x)(x)(1)$. Here $\alpha$ is a primitive fourth root of unity.
			
			Let
			\[
			A_1 =
			\begin{pmatrix}
				0 & 0 & 0 & x \\
				y & 0 & 0 & 0 \\
				0 & z & 0 & 0 \\
				0 & 0 & 1 & 0
			\end{pmatrix};
			(A_1)^4 =
			\begin{pmatrix}
				xyz & 0 & 0 & 0 \\
				0 & xyz & 0 & 0 \\
				0 & 0 & xyz & 0 \\
				0 & 0 & 0 & xyz
			\end{pmatrix}
			=xyz I_4
			\]

			\[
			A_2 =
			\begin{pmatrix}
				A_1 & 4 I_4 & 0 & 0 \\
				0 & \alpha A_1 & x I_4 & 0 \\
				0 & 0 & \alpha^2 A_1 & x I_4 \\
				I_4 & 0 & 0 & \alpha^3 A_1
			\end{pmatrix}
			(A_2)^2 = \begin{pmatrix}
				A_1^2 & 4(1+\alpha)A_1 & 4x I_4 & 0 \\
				0 & \alpha^2 A_1^2 & x(\alpha + \alpha^2)A_1 & x^2 I_4 \\
				x I_4 & 0 & \alpha^4 A_1^2 & x(\alpha^2 + \alpha^3)A_1 \\
				(1 + \alpha^3)A_1 & 4 I_4 & 0 & \alpha^6 A_1^2
			\end{pmatrix}
			\]
			
			\[
			  (A_2)^4 = \begin{pmatrix}
			  	A_1^4 + 4x^2 I_4 & 4\epsilon A_1^3 & 4x(\epsilon+\alpha^2+\alpha^4)A_1^2 & 4x^2 \epsilon A_1 \\
			  	x^2 \epsilon A_1 & \alpha^4 A_1^4 + 4x^2 I_4 & x(\alpha+\alpha^2+\alpha^5+\alpha^6)A_1^3 & x^2(\alpha^2 \epsilon +\alpha^4+\alpha^6)A_1^2 \\
			  	x(1+\alpha^2\epsilon +\alpha^6)A_1^2 & 4x \epsilon A_1 & \alpha^8 A_1^4 + 4x^2I_4 & x(\alpha^6 \epsilon )A_1^3 \\
			  	(1+\alpha^3+\alpha^6+\alpha^9)A_1^3 & 4(\epsilon +\alpha^4+\alpha^6)A_1^2 & 4x \epsilon A_1 & \alpha^{12} A_1^4 + 4x^2 I_4
			  \end{pmatrix}
			\]
			with $\epsilon = 1+\alpha+\alpha^2+\alpha^3$.

			$A_2$ is indeed a matrix fourth root of $f$.
	\end{example}
	
	\begin{proposition}
	If $fI_p=A_1A_2 \cdots A_n$ is a matrix factorisation of $f$ with $n$ matrix factors, then there exists a matrix $B$ satisfying $fI_{np} = B^n$.
	\end{proposition}

	\begin{proof}
	Let $B = (B_{ij})$ be the bloc matrix of size $np$ where
		\[
		\begin{cases}
			B_{n,1} = A_n \\
			B_{i,i+1} = A_i \quad \text{pour } 1 \leq i \leq n-1 \\
			0 \quad \text{everywhere else.}
		\end{cases}
		\]
		\begin{enumerate}
			\item Let $l$ be a non zero natural number.
			\[
			\begin{aligned}
				B_{l,l}^{(n)} &= \sum_{k=1}^{n} B_{l,k} B_{k,l}^{(n-1)} \\
				&= B_{l,l+1}\, B_{l+1,l}^{(n-1)} \\
				&= B_{l,l+1}\, B_{l+1,l+2}\, B_{l+2,l}^{(n-2)} \\
				&= B_{l,l+1}\, B_{l+1,l+2} \cdots B_{l+(n-l),l}^{(n-(n-l))} \\
				&= \left( \prod_{p=l}^{n-1} B_{p,p+1} \right) B_{l+(n-l),l}^{(n-(n-l))} \\
				&= \left( \prod_{p=l}^{n-1} B_{p,p+1} \right) B_{n,l}^{(l)} \\
				&= (\cdots) B_{n,l}^{(l)} \\
				&= (\cdots) B_{n,1} B_{1,l}^{(l-1)} \\
				&= (\cdots) B_{n,1} \left( \prod_{p=1}^{l-2} B_{p,p+1} \right) B_{l-1,l}^{(1)} \\
				&= (\cdots) B_{n,1} \left( \prod_{p=1}^{l-2} B_{p,p+1} \right) B_{l-1,l} \\
			    &= (\cdots) B_{n,1} (\cdots) B_{l-1,l} \\
			    &= A_l A_{l+1} \cdots A_{n-1} A_n A_1 A_2 \cdots A_{l-2} A_{l-1} \\
			    &= fI_p
			\end{aligned}
			\]
			
			\item Let $l,s$ be non zero natural numbers such that $l \neq s$
			\[
			\begin{aligned}
		        B_{l,s}^{(n)} &= \sum_{k=1}^{n} B_{lk} B_{k,s}^{(n-1)} \\
		                      &= (\cdots) B_{n,1} (\cdots) B_{l-1,s}
			\end{aligned}
			\]
			and $B_{l-1,s}=0$ (the zero matrix) by definition of the matrix $B$.
			
			Hence,
			\[
			(B)^n = I_n \otimes (fI_p) = fI_{np}
			\]
		\end{enumerate}	
	\end{proof}



\bibliography{fomatati_ref}
\addcontentsline{toc}{section}
{References}

\end{document}